\documentclass[a4paper, reqno]{amsart}
\usepackage{amsmath, amssymb, amscd, amsthm, mathtools, comment} 
\usepackage{graphicx, multicol, caption}
\usepackage{tikz,extarrows,graphicx,color}

\newenvironment{Figure}
  {\par\medskip\noindent\minipage{\linewidth}}
  {\endminipage\par\medskip}

\newtheorem{definition}{Definition}[section]
\newtheorem{proposition}{Proposition}[section]
\newtheorem{lemma}{Lemma}[section]
\newtheorem{theorem}{Theorem}[section]
\theoremstyle{remark}

\begin{document}

\title[Random Fibonacci]{Topology of the Random Fibonacci Tiling Space}

\author[F. G\"ahler]{Franz G\"ahler\textsuperscript{a}}
\author[E. Provido]{Eden Provido\textsuperscript{a,b}}
\address{Fakult\"at f\"ur Mathematik, Universit\"at Bielefeld, 33615 Bielefeld, Germany\textsuperscript{a}}
\email{\{eprovido, gaehler\}@math.uni-bielefeld.de}
\address{Department of Mathematics, Ateneo de Manila University, 1108 Quezon City, Philippines\textsuperscript{b}}

\begin{abstract} 
We look at the topology of the tiling space of locally random Fibonacci substitution, which is
defined as $a\mapsto ba$ with probability $p$, $a\mapsto ab$ with probability $1-p$ and $b\mapsto a$ for $0<p<1$.
We show that its \v{C}ech cohomology group is not finitely generated, in contrast to the case where random substitutions are applied globally.

\end{abstract}

\maketitle

\begin{multicols}{2}

\section{Introduction}

In 1998, Anderson and Putnam showed that a substitution tiling space can be represented as inverse limit of an inflation and substitution map on the cellular Anderson-Putnam complex \cite{ap98}. 
The cohomology of the tiling space can then be computed as the direct limit of the homomorphism induced by inflation and substitution on the cohomology of the complex. 
G\"{a}hler and Maloney \cite{gm13} investigated the cohomology of one-dimensional tiling spaces arising from several substitutions which are all acting on the same set of tiles. 
Since each substitution is applied globally, the entropy remains zero and the cohomology remains finitely generated.

The next natural goal is to go one step further by studying mixed substitutions that are randomly applied locally, which makes the entropy positive \cite{gl89}. 
In this study, we use the concept of Rauzy graphs which, as shown in \cite{j10}, is a very convenient tool to relate complexity and cohomology.
For simplicity, we illustrate our method using the random Fibonacci substitution even though it also works in more general cases, such as the random noble means substitutions, which were introduced and studied in 
\cite{bm13, m14}. They are given by
$$ \zeta_m: \left\{
  \begin{array}{lll}
    a & \mapsto & \left\{
                    \begin{array}{ll}
                      ba^m & \mbox{with probability} \phantom{0} p_0 \\
                      aba^{m-1}& \mbox{with probability} \phantom{0} p_1 \\
                      \phantom{LSL^{m-1}} \vdots &  \phantom{0}  \\
                      a^mb & \mbox{with probability} \phantom{0} p_m
                    \end{array}
                  \right.
     \\
    b & \mapsto & a
  \end{array}
\right.$$
where $(p_0,\ldots,p_m)$ is a fixed probability vector, that is $p_i \geq 0$ and $\sum^{m}_{i=0} p_i =1$.
The random Fibonacci substitution is the case where $m=1$. It was introduced by Godr\`{e}che and Luck \cite{gl89} in 1989,
and, together with its generalised form, has been analyzed intensively by Nilsson \cite{n10}.

\section{Preliminaries} \label{sec:intro}

A \emph{tile} is a subset of $\mathbb{R}^n$ that is homeomorphic to a closed unit ball.
A \emph{partial tiling} $T$ is a set of tiles, any two of which intersect only on their boundaries, and the \emph{support of }$T$ is the union of its tiles.
We define a \emph{tiling} as a partial tiling whose support is $\mathbb{R}^n$. 

Any set of tilings of $\mathbb{R}^n$  can be equipped with a metric, in which two tilings are close if, up to a small translation, they agree on a large ball around the origin.
We use the metric defined in \cite{ap98}: for any two tilings $T$ and $T'$ of $\mathbb{R}^n$,
\begin{eqnarray}
 d(T,T') & := & \mbox{inf} \big( \big\{ 1/\sqrt{2} \big\} \cup  \big\{ \epsilon \mbox{\phantom{0}}|\mbox{\phantom{0}} T+u \mbox{ and } T'+v      \nonumber \\
   & & \mbox{ agree on } B_{1/\epsilon}(0) \mbox{ for some } \|u\|,\|v\|<\epsilon \big\} \big) \nonumber 
\end{eqnarray}
where $\|\cdot\|$ is the usual norm on $\mathbb{R}^n$ and $B_r(x)$ is the open ball of radius $r$ centered at $x$ in $\mathbb{R}^n$.
The topology arising from this metric may be generated by the cylinder sets
$$Z_\epsilon (T) = \{T'\in \Omega \mbox{\phantom{0}}|\mbox{\phantom{0}} d(T,T')< \epsilon\}.$$
This means that if $T'$ is in the cylinder set $Z_\epsilon (T)$, then $T$ and $T'$ agree on the ball $B_{1/\epsilon}(0)$ up to translation by $\epsilon$.

\section{The Random Fibonacci Tiling Space}
In this study, we look at a set of tilings in $\mathbb{R}$.
Let $a$ and $b$ be closed intervals of lengths $\lambda=\left ( 1+\sqrt{5}\right)/2$ and 1, respectively. 
Let $\{a,b\}$ be the set of prototiles and let $\tilde{\Omega}$ denote the set of all partial tilings that contain only translates of these prototiles.

Let $0<p<1$. 
Define the random Fibonacci inflation rule as
$$ \zeta: \left\{
  \begin{array}{lll}
    a & \mapsto & \left\{
                    \begin{array}{ll}
                      ba & \mbox{with probability} \phantom{0} p \\
                      ab & \mbox{with probability} \phantom{0} 1-p 
                    \end{array}
                  \right.
     \\
    b & \mapsto & a
  \end{array}
\right. .$$
We say that it is a \emph{local random inflation} because we apply $\zeta$ locally, that is, for each occurrence of the tile $a$ we decide independently which among the $2$ possible realizations to choose.

In particular, $\zeta$ is an inflation map from $\{a,b\}$ to $\tilde{\Omega}$ with inflation constant $\lambda$ where the supports of $\zeta(a)$ and $\zeta(b)$ are $\lambda a$ and $\lambda b$, respectively.
Then $\zeta$ can be extended to a map
$\zeta: \tilde{\Omega} \rightarrow \tilde{\Omega}$
by
$$\zeta(T)=\bigcup_{t+r \in T}(\zeta(t)+\lambda r)$$
where $t\in \{a,b\}$ and $r\in \mathbb{R}$.

The \emph{random Fibonacci tiling space} $\Omega$ is the set of all tilings $T\in \tilde{\Omega}$ such that for any partial tiling $P\subseteq T$ with bounded support, 
we have $P\subseteq \zeta^n(b+r)$ for some $r\in \mathbb{R}$.

The following result is of interest because in the case of the global random substitution \cite{gm13}, for two distinct substitution sequences $s_1$ and $s_2$
the corresponding hulls may be disjoint.

\begin{proposition}
$\Omega$ is connected.
\end{proposition}

\begin{proof}
Suppose that $\Omega$ is not connected.
Then there exist open sets $U,V$ with $U \cap V = \emptyset$ and $\Omega=U \cup V$.

Let $T_U\in U$ and $T_V \in V$.
Since $T_U\in U$ and $U$ is open, there is  
a cylinder set $Z_{1/R}(T_U)$ contained in $U$, for some $R>0$; that is, $T_U \in Z_{1/R}(T_U) \subset U$.
Choose some number $r>R$, and let $P=[B_r(0)]$ be a partial tiling in $T_U$ consisting of all tiles that intersect the ball $B_r(0)$.
It is easy to see that $P$ is also contained somewhere in $T_V$.
Then some translates of $T_V$ agree with $T_U$ on the ball $B_r(0)$.
Consider one such translate of $T_V$, say $T_V-x$ for some $x\in \mathbb{R}$.
Since $r>R$, $T_V-x$ and $T_U$ also agree on the ball $B_R(0)$, which means that $T_V-x \in Z_{1/R}(T_U) \subset U$.

Now $T_V -x \in \mathcal{O}(T_V)$, the translation orbit of $T_V$, which lies in a path component of $\Omega$.
Each path component of $\Omega$ lies in a component of $\Omega$.
Also since a component of $\Omega$  is a connected subspace of $\Omega$, it lies entirely within either $U$ or $V$.
Since $T_V \in V$, the path component containing $\mathcal{O}(T_V)$ must lie in $V$.
Hence $T_V-x \in \mathcal{O}(T_V) \subset V$.

We now have $T_V-x \in U$ and $T_V-x \in V$, which contradicts the assumption that $U \cap V = \emptyset$.
Thus $\Omega$ is connected. 
\end{proof}

\section{Symbolic Random Fibonacci Tilings}

Another way to look at one-dimensional tiling spaces is in terms of symbolic dynamics.


Let our alphabet be $\mathcal{A}=\{a,b\}$. 
A \emph{word} $w$ of length $n$ over $\mathcal{A}$ is a sequence $w_1 w_2 \ldots w_n$ of $n$ letters from $\mathcal{A}$. 
If $u=u_1 u_2 \ldots u_n$ and $v=v_1 v_2 \ldots v_m$ are two words, then we let $uv$ denote the \emph{concatenation} of $u$ and $v$; that is, $uv=u_1 u_2 \ldots u_n v_1 v_2 \ldots v_m$. 
Similarly, for two sets of word $U$ and $V$, we let their product be the set $UV=\{uv:u\in U, v\in V\}$ which contains all possible concatenations. 
We define the set of \emph{bi-infinite words} over $\mathcal{A}$ to be $\mathcal{A}^\mathbb{Z} = \{(w_i)_{i\in \mathbb{Z}}:w_i \in \mathcal{A}\}$, and 
let $S:\mathcal{A}^\mathbb{Z}\rightarrow \mathcal{A}^\mathbb{Z}$ be the \emph{shift operator}, $S((w_i)_{i\in \mathbb{Z}})=(w_{i+1})_{i\in \mathbb{Z}}$.

Consider the substitution $\zeta$, where $a$ and $b$ are letters instead of intervals. 
As before, $\zeta$ is randomly applied locally.
Define the \emph{symbolic random Fibonacci tiling space} $\Sigma$ as the set of all elements of $\mathcal{A}^\mathbb{Z}$ such that for any $w \in \Sigma$ and for any subword $u$ of $w$, 
there exists a power $k\in \mathbb{N}$ 
where $u$ is a subword of some realisation of $\zeta^k(b)$. Clearly, $\Sigma$ is totally disconnected, compact, and shift-invariant.

Now, define the \emph{suspension} $(\Omega,\phi)$ of $(\Sigma,S)$ by
$$\Omega=\Sigma \times \mathbb{R}/\sim,$$
where $\sim$ is the equivalence relation generated by the relations $(w,t)\sim (Sw,t-1)$ for all $w\in \Sigma$ and all $t\in \mathbb{R}$.
The map $\phi$ is an action of $\mathbb{R}$ defined by
$\phi_s([(w,t)]):=[(w,t+s)].$
We have that $(\Omega,\phi)$ is an $\mathbb{R}$-dynamical system.

It is easy to see that $\Omega$ is homeomorphic to the random Fibonacci tiling space defined previously, justifying the use of the same notation.
Moreover, the cohomology groups do not depend on the length of the tiles, which allows us to work on the symbolic level.

\section{Rauzy Graphs and their Inverse Limit}
Before we can define a Rauzy graph, we need a finite method for finding the set of subwords of the words in the symbolic random Fibonacci tiling space $\Sigma$.

If we let $\zeta$ act on letter $b$ repeatedly, it yields an infinite sequence of words $r_n=\zeta^{n-1}(b)$. 
We call $r_n$ an \emph{inflated word in generation} $n$, and 
let $A_n$ be the set of all inflated words in generation $n$.
It can be defined recursively as follows: 
let $A_0:=\emptyset$, $A_1:=\{b\}$ and $A_2:=\{a\}$, and for $n\geq 3$,
$A_n:=A_{n-1}A_{n-2}\cup A_{n-2}A_{n-1}$.
Let the \emph{subword set} $F(S,n)$ be the set of all subwords of length $n$ of the words in a set $S$. 
In \cite{n10}, it is shown that, for $n\geq 4$, $$F(A_{n+1},f_n)=F(\Sigma,f_n)$$
where the numbers $f_n$ are Fibonacci numbers.

Let us now define the Rauzy graphs associated with $\Sigma$.
The following are based on the definitions and results in \cite{j10}.

\begin{definition}
 Let $F_n:=F(\Sigma,n)$, the set of all subwords of length $n$ of the words in $\Sigma$. 
 The $n$th Rauzy graph $R_n$ associated with $\Sigma$ is an oriented graph with set of vertices $V_n=F_n$ and set of oriented edges $E_n$, where 
 the oriented edges are defined as follows: there is an edge from $u_1 u_2 \ldots u_n \in V_n$ to $v_1 v_2 \ldots v_n \in V_n$ if $u_i=v_{i-1}$
 for all $2\leq i \leq n$ and $u_1 u_2 \ldots u_nv_n \in F_{n+1}$.
\end{definition}

For any two vertices $u$ and $v$, there is a path from $u$ to $v$ and a path from $v$ to $u$ as well, and
so the Rauzy graphs $R_n$ associated with $\Sigma$ are connected by oriented paths, that is, they are \emph{strongly connected}.

Moreover, the Rauzy graphs are not just combinatorial objects but are also metric spaces. 
We will consider $R_n$ as the disjoint union of copies of the interval $[0,1]$ indexed over the edges modulo indentification of endpoints:
$$R_n:=\coprod[0,1]_{(xfy)}/\sim$$
where $xfy\in F_{n+1}$ and $[0,1]_{(xfy)}$ is an edge with tail $xf$ and head $fy$. 
Two endpoints are identified if they stand for the same subword.

We now show that the tiling space $\Omega$ is an inverse limit of Rauzy graphs. 
Clearly, a word of length $n+1$ can be shortened to a word of length $n$.
This fact allows us to define projection maps $R_{n+1}\rightarrow R_n$, which we
need to build inverse limits.

\begin{definition} Define $\gamma_n: R_{n+1}\rightarrow R_n$ in the following way.

\noindent If $n$ is even, $\gamma_n$ is defined by
\begin{enumerate}
 \item $\gamma_n(u:=u_1 u_2 \ldots u_{n+1})=u_1 u_2 \ldots u_n$ if $u\in V_{n+1}$; 
 \item if an edge $e\in E_{n+1}$ goes from $u$ to $v$, then there exists in $R_n$ an edge from $\gamma_n(u)$ to $\gamma_n(v)$;
       the application $\gamma_n$ then maps $e$ to that edge, and its restriction on $e$ is the identity map on the segment $[0,1]$.
\end{enumerate}
If $n$ is odd, $\gamma_n$ is similarly defined by
\begin{enumerate}
 \item $\gamma_n(u:=u_1 u_2 \ldots u_{n+1})=u_2 \ldots u_{n+1}$ if $u\in V_{n+1}$;
 \item the definition of $\gamma_n$ on the edges is exactly the same as above.
\end{enumerate}
\end{definition}

\begin{Figure}
\centering
\begin{tabular}{ c c }
\begin{tabular}{ c }
  $R_1$ \\
  $p(1)=2$ \\
  $s(1)=2$ \\
\end{tabular} &
\begin{tabular}{ c }
  \includegraphics[scale=.7]{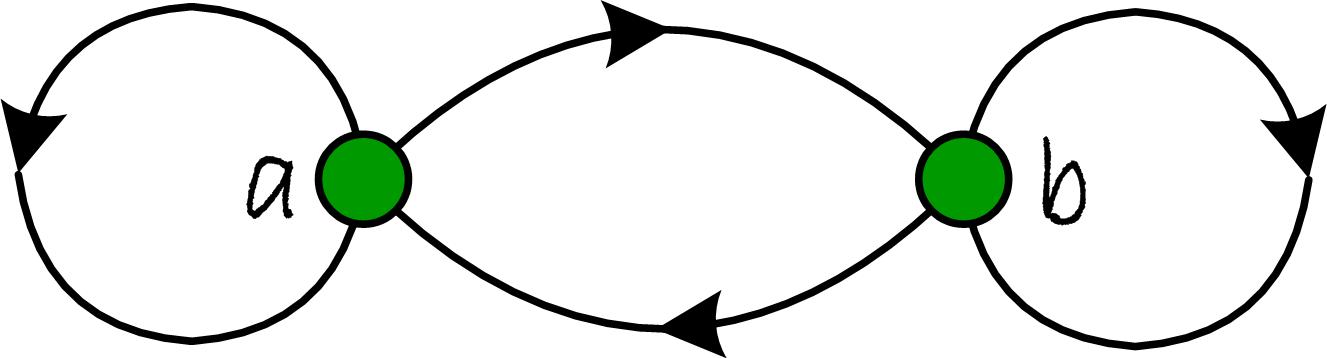} \\
\end{tabular}
  \\
\begin{tabular}{ c }
  $R_2$ \\
  $p(2)=4$ \\
  $s(2)=3$ \\
\end{tabular} &
\begin{tabular}{ c }
  \includegraphics[scale=.7]{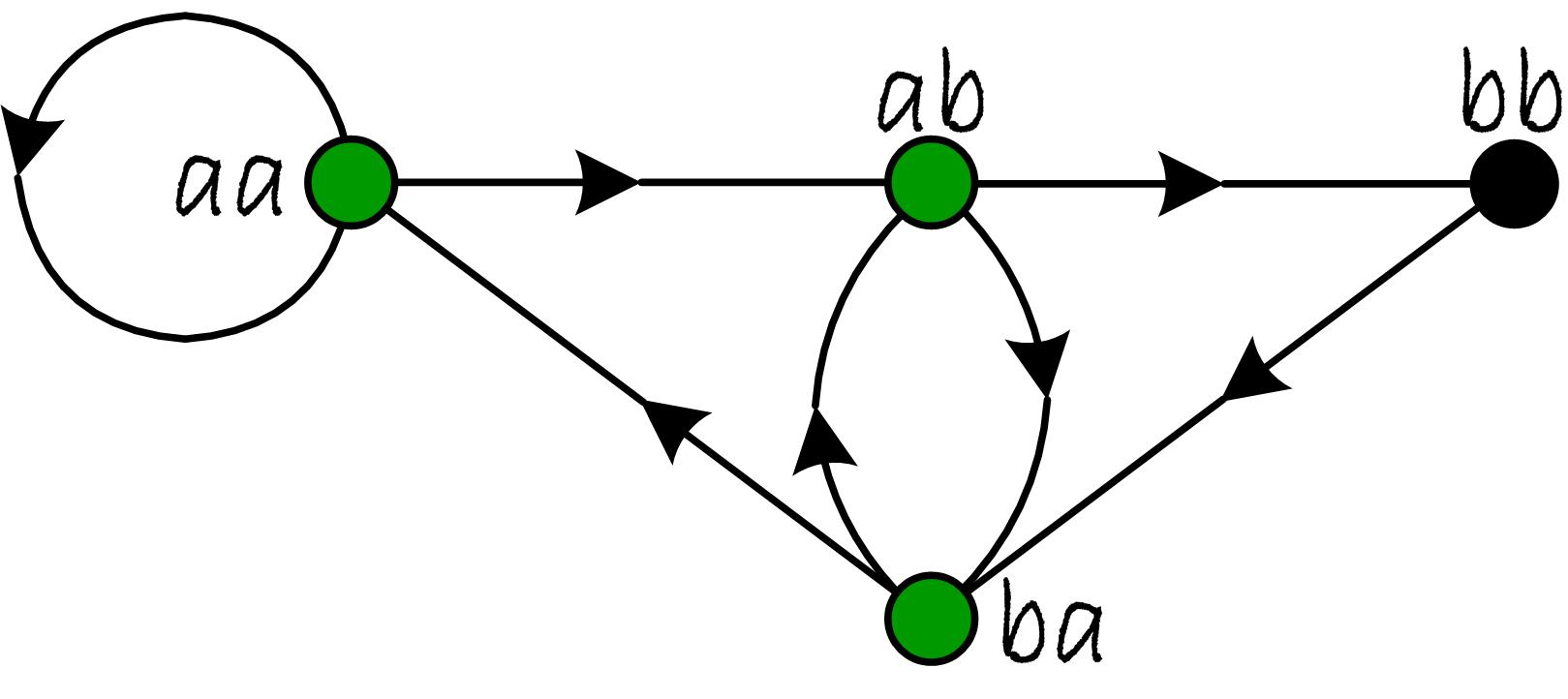} \\
\end{tabular}
  \\
\begin{tabular}{ c }
  $R_3$ \\
  $p(3)=7$ \\
  $s(3)=6$ \\
\end{tabular} &
\begin{tabular}{ c }
  \includegraphics[scale=.7]{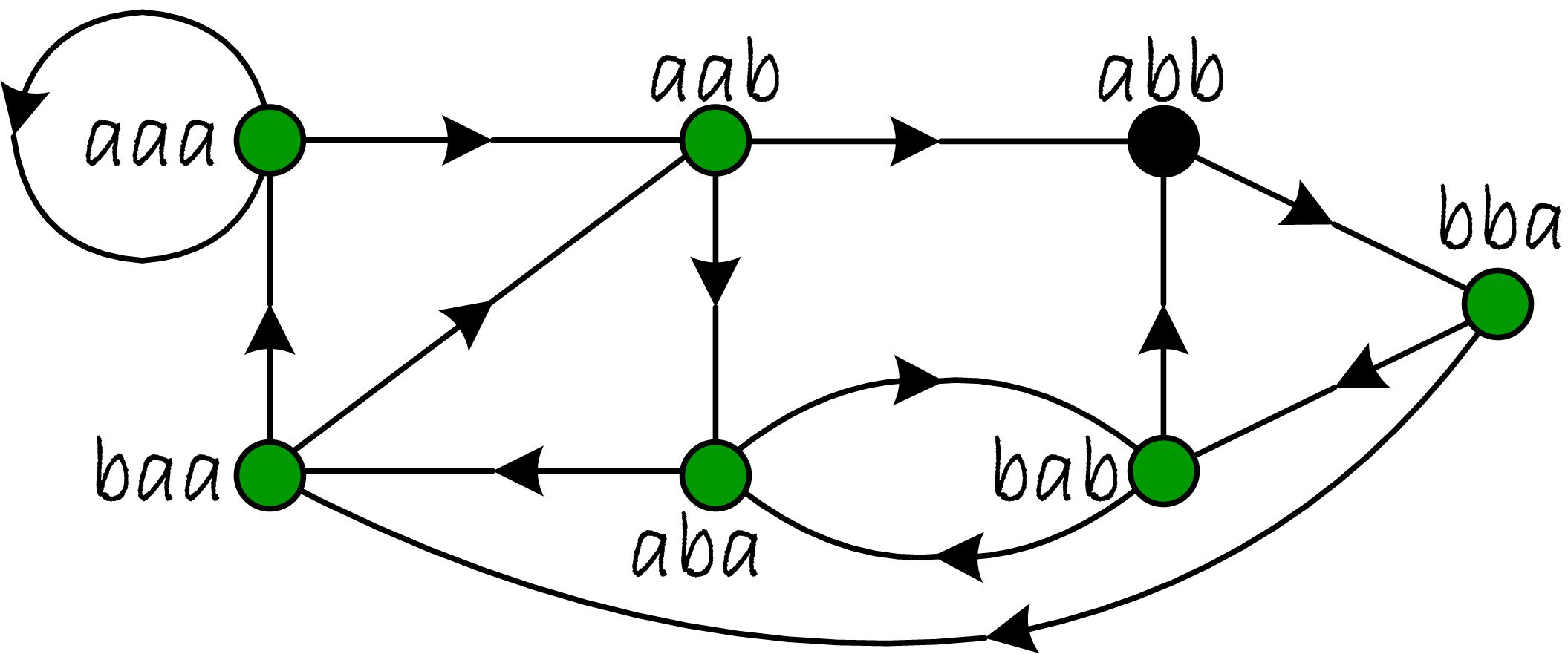} \\
\end{tabular}
  \\
\end{tabular}
\captionof{figure}{For $n=1,2,3$, we illustrate the Rauzy graphs $R_n$, and also give their corresponding complexity $p(n)$ and the first difference of their complexity 
$s(n)$ both of which are introduced in Section 6. The grey nodes correspond to the right special subwords.}
\end{Figure}

The proof of the following proposition is straightforward.

\begin{proposition}\label{qou}
The maps $\gamma_n$ are surjective, continuous, open maps. 
\end{proposition}

Let us now consider the inverse limits of the Rauzy graphs with respect to the maps $\gamma_n$.

\begin{definition}
The inverse limit of the complex $(R_n,\gamma_n)_{n\in \mathbb{N}}$ is the following subset of $\prod_{n\in \mathbb{N}}R_n$:
$$\Big \{ (x_1,x_2,\ldots)\in \prod_{n\in \mathbb{N}}R_n \; \Big \vert \; \forall n\in \mathbb{N}, \gamma_n(x_{n+1})=x_n  \Big \}$$
which is denoted by ${\displaystyle\lim_{\longleftarrow}(R_n,\gamma_n)_{n\in \mathbb{N}}}$.
\end{definition}

The inverse limit is a topological space with the relative topology induced by the product topology of $\prod_{n\in \mathbb{N}}R_n$. 
It is a compact set, being a closed subset of a compact set.

The following theorem has been proved in \cite{j10} for the one-dimensional standard substitution tiling spaces. 
That argument can be used with very minimal modification to prove the case for the random Fibonacci tiling space $\Omega$.

\begin{theorem}[Theorem 6.10 in \cite{j10}] \label{invlim}
Let $\Omega$ denote the suspension of $(\Sigma,S)$. 
Then we have the following homeomorphism:
$$\Omega \cong \lim_{\longleftarrow}(R_n,\gamma_n)_{n\in \mathbb{N}}.$$
\end{theorem}

\section{Complexity and Cohomology}
As before, we denote by $F_n=F(\Sigma,n)$ the set of all subwords of length $n$ of the words in $\Sigma$, 
and denote by $F$ the union of $F_n$ for all $n$.
Note that $F$ contains all the subwords of any length of the words in $\Sigma$, 
and for every word $v\in F$ there is a pair of letters $(x,y)$ such that $xvy \in F$.
We define the subword complexity function $p(n)$ as the cardinality of the set $F_n$, 
and let the first difference of the complexity be denoted by $s(n):=p(n+1)-p(n)$.

The number of vertices of the $n$th Rauzy graph $R_n$ is equal to the complexity $p(n)$, 
and its number of edges is $p(n)+s(n)$.

The following proposition from \cite{j10} makes an explicit link between $s(n)$ and the cohomology of the $n$th Rauzy graph $R_n$.

\begin{proposition}[Proposition 6.14 in \cite{j10}] \label{rk}
 The rank of the cohomology $H^1(R_n;\mathbb{Q})$ of the $n$th Rauzy graph equals $s(n)+1$.
\end{proposition}

We now introduce some useful concepts defined in \cite{c95}.
Every element of length $n$ in $F$ can be extended  in at least one way to an element of length $n+1$ by appending a letter to the right,
we call those elements for which this can be done in two different ways as \emph{right specials}.
Similarly, the words $v\in F$ such that $av \in F$ and $bv \in F$ are called \emph{left specials}.
It is easy to see that $s(n)$ may be interpreted as the number of right special elements of length $n$ in $F$, or
as the number of left special elements of length $n$ in $F$.

If an element $v$ of $F$ is both right special and left special, we say that $v$ is \emph{bispecial}.
There are three kinds of bispecial elements, depending  on the cardinality of the set $I=F\cap (\mathcal{A}v\mathcal{A})$.
When $I$ has four elements, $v$ is a \emph{strong bispecial}, and when it has only two elements, $v$ is a \emph{weak bispecial}.
If we let $sb(n)$ denote the number of strong bispecials of length $n$ and $wb(n)$ the number of weak bispecials of length $n$, then
it can be shown that $s(n+1)-s(n)=sb(n)-wb(n)$.

The following lemma implies that, for the random Fibonacci tiling space $\Omega$, there are no weak bispecials in $F$.
Thus, we have $s(n+1)-s(n)=sb(n)$, which tells us that $s(n)$ is nondecreasing.

\begin{lemma} \label{wb}
Suppose $v \in F$ such that both $va$ and $vb$ ($av$ and $bv$) are in $F$. 
Then, $va$ and $vb$ ($av$ and $bv$) can be legally extended to the left (right) both with an $a$, or both with a $b$.
That is, either $ava,avb \in F$ or $bva,bvb \in F$ ( $ava,bva \in F$ or $avb,bvb \in F$).
\end{lemma}

\begin{proof}
We prove the case where $v$ is first extended to the right. The other case is analogous.

As a legal word, $v$ occurs somewhere in an infinite word $w$ coming from a substitution. 
Thus, the word $w$ has a partitioning into first order supertiles $[ab]$, $[ba]$, and $[a]$. 
Different occurrencies of $v \subset w$ may have different partionings into supertiles, 
but if $v$ is extensible to the right in two ways, there must be an occurrence of $v \subset w$ which ends on a supertile boundary. 
Otherwise, only one continuation would be possible. 
As $v$ can be extended to the right in two ways, it must be possible 
to extend it with a supertile $[ab]$, and thus also with $[ba]$. 
If only $[a]$ was possible, there would only be one way to extend. 
Now take an infinite word $w$ with a fixed occurrence of $v[ab]$. 
As $w$ is random substitution word, we can swap the pair $[ab]$
to the right of $v$, leaving the rest of $w$ unchanged, and obtain 
another legal word $w'$, containing $v[ba]$ instead. As to the left 
of $v$ both $w$ and $w'$ agree, the claim follows.
\end{proof}

In Theorem \ref{invlim}, we have seen that the suspension space $\Omega$ is homeomorphic to the inverse limit of the Rauzy graphs.
This implies that the $\check{\mbox{C}}$ech cohomology of $\Omega$ is the direct limit of the cohomology of the Rauzy graphs.
In particular, the $\check{\mbox{C}}$ech cohomology group $H^1(\Omega)$ is isomorphic to the direct limit of the system of abelian groups
$$H^1(R_1)\xlongrightarrow{\gamma_1^*}H^1(R_2)\xlongrightarrow{\gamma_2^*}\cdots\xlongrightarrow{\gamma_{n}^*}H^1(R_{n+1})\xlongrightarrow{\gamma_{n+1}^*}\cdots.$$

We want to show is that for the random Fibonacci tiling space $\Omega$, the maps $\gamma_n^*: H^1(R_n)\rightarrow H^1(R_{n+1})$ are injective. 
The proof of this uses the theory of quotient cohomology, which was introduced in \cite{bs11}.
We now introduce some notation relating to $\check{\mbox{C}}$ech cohomology. 
If $X$ is a topological space with the structure of a CW-complex, let $c$ denote a cell in $X$,
and let $c'$ denote the corresponding cochain. Let $C^d(X)$ denote the group of $d$-cochains of $X$,
and $\delta_d:C^d(X) \rightarrow C^{d+1}(X)$ denote the coboundary map. 

The theory of quotient cohomology applies to topological spaces $X$ and $Y$ for which there is a quotient map $f:X\rightarrow Y$
such that the pullback $f^*$ is injective on cochains. 
Define the cochain group $C^d_Q(X,Y)$ to be the quotient $C^d(X)/f^*(C^d(Y))$, and the
let the quotient coboundary operator to be $\bar{\delta}_d: C^d(X,Y) \rightarrow C^{d+1}(X,Y)$.
Then the quotient cohomology between $X$ and $Y$ is defined to be $H^d_Q(X,Y):=\mbox{ker} \bar{\delta}_d/\mbox{im} \bar{\delta}_{d-1}$.

Then, by the snake lemma, the short exact sequence of cochain complexes
$$0\longrightarrow C^d(Y) \xlongrightarrow{f^*} C^d(X)\longrightarrow C^d(X,Y)\longrightarrow 0$$
induces a long exact sequence
$$\rightarrow H_Q^{d-1}(X,Y)\rightarrow H^{d}(Y) \xrightarrow{f^*} H^{d}(X)\rightarrow H_Q^{d}(X,Y)\rightarrow \cdots.$$

\bigskip

\begin{proposition} \label{inj}
 If $R_n$ are the Rauzy graphs of the random Fibonacci tiling space $\Omega$, the induced maps $$\gamma_n^*: H^1(R_n)\rightarrow H^1(R_{n+1})$$ are injective. 
\end{proposition}

\begin{proof}
The following proof is adapted from \cite{gm13}.

 Let $X=R_{n+1}$ and $Y=R_n$. 
 By Proposition \ref{qou}, the map $\gamma_n: R_{n+1} \rightarrow R_n$ is a quotient map, and the pullback of this map
 is injective on cochains. Then we have the following exact sequence
 $$0\rightarrow H^{0}(Y)\rightarrow H^{0}(X) \rightarrow H_Q^{0}(X,Y)\rightarrow$$
$$H^{1}(Y)\xlongrightarrow{\gamma_n^*} H^{1}(X) \rightarrow H_Q^{1}(X,Y)\rightarrow 0.$$
 To prove that  $\gamma_n^*$ is injective, we show that $H_Q^{0}(X,Y)=0$.
 To this end, we show that if $x\in C^0(X)$ such that $\bar{\delta}_0(x)\in \gamma_n^*(C^1(Y))$, then $x\in \gamma_n^*(C^0(Y))$.
 
 Suppose $n$ is odd. Then $\gamma_n(u:=u_1 u_2 \ldots u_{n+1})=u_2 \ldots u_{n+1}$. 
 In the following, we use these notations: for some $h\in F_m$, $(*h)':=\sum_{th \in F_{m+1}}th$ and $(h*)':=\sum_{ht \in F_{m+1}}ht$.
 The image of a cochain $h'$ under $\gamma_n^*$ is $(*h)'$.
 Furthermore, supposing $x\in C^0(X)$ and $\bar{\delta}_0(x)\in \gamma_n^*(C^1(Y))$,
 we have $$x=\sum_{rs\in F_{n+1},s\in F_n} n_{rs}(rs)'$$
 and
 $$\bar{\delta}_0(x)=\sum_{uv\in F_{n+1},u\in F_n} m_{uv}(*uv)'$$
 for some $n_{rs},m_{uv}\in \mathbb{Z}$.
 On the other hand, by definition of $\bar{\delta}_0$, we have
 \begin{eqnarray*}
 \bar{\delta}_0(x) & = & \bar{\delta}_0\left(\sum n_{rs}(rs)'\right)=\sum n_{rs}\delta_0(rs)'  \\
                   & = & \sum n_{rs}\left((*rs)'-(rs*)'\right)
\end{eqnarray*}
 where summations are taken over $rs\in F_{n+1},s\in F_n$.
 Taking modulo $\gamma_n^*(C^1(Y))$, we get
\begin{eqnarray}
 \bar{\delta}_0(x) &=& -\sum_{rs\in F_{n+1},s\in F_n} n_{rs}(rs*)' \nonumber \\
   &=& \mbox{\phantom{--.}}\sum_{uv\in F_{n+1},u\in F_n} m_{uv}(*uv)' \nonumber 
\end{eqnarray}
By Lemma \ref{wb}, for some $s \in F$, if $as$ and $bs$ are elements of $F$, then $asa,bsa \in F$ or $asb,bsb \in F$. 
The first case implies that $-n_{as}=m_{sa}$ and $m_{sa}=-n_{bs}$, while from the second case we have $-n_{as}=m_{sb}$ and $m_{sb}=-n_{bs}$.
Thus, in particular, the coefficients $n_{rs}$ depend on $s\in F_n$ alone, and so we have $$x=\sum_{s\in F_n} n_{s}(*s)'.$$
Hence, $x$ lies in $\gamma_n^*(C^0(Y))$.

Similar arguments can be used for the case where $n$ is even.
\end{proof}
 
In \cite{j10}, Julien showed that if the complexity function $p(n)$ is bounded by $Cn$ for some constant $C$,
then the cohomology is rationally finitely generated. 
This is the case for the global random substitution \cite{gm13} where $p(n)$ is bounded, and
hence the cohomology is also rationally finitely generated.

However, in our case the system has positive entropy which means that $p(n)$ grows exponentially, and
$s(n)$ is nondecreasing. 
By Proposition \ref{rk} and Proposition \ref{inj}, the rank of the cohomology of $R_n$ grows without bound,
and the kernels of the bonding maps between them vanish. As a consequence, the direct limit cannot be
finitely generated and we can deduce the following.

\medskip

\begin{theorem}
The \v{C}ech cohomology group $H^1(\Omega)$ is not finitely generated.
\end{theorem}

\section*{Acknowledgement}

The authors would like to thank Paolo Bugarin, Greg Maloney, and Johan Nilsson for fruitful discussions. 
E.~Provido would also like to thank the Alexander von Humboldt Foundation for the generous support.
This work is also supported by DFG via the Collaborative Research Centre 701 and the 
Research Center for Mathematical Modelling (RCM\textsuperscript{2}) at Bielefeld
University.

\end{multicols}
\end{document}